\documentclass{article}

\usepackage[utf8]{inputenc}
\usepackage{tikz}
\usepackage{amssymb,amsmath,caption,subcaption,graphicx,hyperref}
\usepackage{amsthm}
\usepackage{authblk, blindtext}
\usepackage{lineno}
\usetikzlibrary{decorations.pathreplacing}

\newlength\tindent
\newtheorem{thm}{Theorem}[section]

\newtheorem{prop}[thm]{Proposition}
\newtheorem{cor}[thm]{Corollary}
\newtheorem{lem}[thm]{Lemma}

\newtheorem{ex}[thm]{Example}

\newtheorem{obs}[thm]{Observation}

\newcommand{\bpf}{\begin{proof}}
\newcommand{\epf}{\end{proof}}

\newcommand{\dsp}{\displaystyle}

\newcommand{\aS}{\operatorname{rb}}

\definecolor{ao(english)}{rgb}{0.0, 0.5, 0.0}

\definecolor{brown(traditional)}{rgb}{0.59, 0.29, 0.0}

\title{Rainbow numbers of $[n]$ for $\sum_{i=1}^{k-1} x_i = x_k$.}

\author[1]{Kean Fallon}
\author[2]{Colin Giles}
\author[3]{Hunter Rehm}
\author[1]{Simon Wagner}
\author[1]{Nathan Warnberg}

\affil[1]{Department of Mathematics and Statistics, University of Wisconsin-La Crosse, \{fallon.kean,wagner.simon,nwarnberg\}@uwlax.edu}

\affil[2]{Department of Mathematics, Iowa State University, cgiles@iastate.edu}

\affil[3]{Department of Mathematics and Statistics, University of Vermont, hunter.rehm@uvm.edu}

\date{March 19, 2020}

\begin{document}

\maketitle
\section*{Abstract}

Consider the set $\{1,2,\dots,n\} = [n]$ and an equation $eq$.  The rainbow number of $[n]$ for $eq$, denoted $\aS([n],eq)$, is the smallest number of colors such that for every exact $\aS([n], eq)$-coloring of $[n]$, there exists a solution to $eq$ with every member of the solution set assigned a distinct color.  This paper focuses on linear equations and, in particular, establishes the rainbow number for the equations $\sum_{i=1}^{k-1} x_i = x_k$ for $k=3$ and $k=4$.  The paper also establishes a general lower bound for $k \ge 5$.\\
{\bf Keywords:} anti-Ramsey, rainbow, Schur

\section{Introduction}





    Let $\{1,2,\dots,n\} = [n]$ and $eq$ be any equation.  The rainbow number of $[n]$ for $eq$, denoted $\aS([n],eq)$, is the smallest number of colors such that for every exact $\aS([n], eq)$-coloring of $[n]$, there exists a solution to $eq$ with every member of the solution set assigned a distinct color.  Although much of the literature has focused on linear equations, as does this paper, this definition is general enough to define rainbow numbers for any set and any equation.  The rainbow number is an anti-Ramsey type function since the key structures are multichromatic (or rainbow) as opposed to monochromatic.  The rainbow number function was inspired by Schur numbers (see \cite{H} and references therein, the authors originally thought of the rainbow number as the anti-Schur number) and the anti-van der Waerden function.  The anti-van der Waerden number on $[n] = \{1,\dots,n\}$, denoted $\operatorname{aw}([n],k)$, is the smallest number of colors such that every exact $\operatorname{aw}([n],k)$-coloring of $[n]$ is guaranteed to have an arithmetic progression of length $k$ where each element of the progression is colored distinctly (note that $3$ term arithmetic progressions satisfy the equation $x_1 + x_2 = 2x_3$). The anti-van der Waerden number was first defined in \cite{U} and many results on arithmetic progressions on $[n]$, the cyclic groups $\mathbb{Z}_n$, finite abelian groups and graphs have been considered directly or indirectly (see \cite{AF,SWY, BSY, DMS, J, RSW, finabgroup}).  Rainbow results have also been considered for equations of the form $a_1x_1 + a_2x_2 + a_3x_3 = b$ over $\mathbb{Z}_p$ (see \cite{BKKTTY, RFC, LM}).

    An \emph{exact $r$-coloring} $c$ of the integers $\{1,2,\dots,n\} = [n]$, is a function $c: [n] \rightarrow [r]$ such that $c$ is onto.  If $eq$ is any linear equation on $k$ variables, namely $x_1,x_2,\dots,x_k$, and non-zero coefficients, then the set $\{s_1,s_2,\dots,s_k\}$ is a \emph{rainbow solution} if substituting $s_i = x_i$ makes $eq$ true and $|\{c(s_1),c(s_2),\dots,c(s_k)\}| = k$.  Note that if a solution does not have distinct elements it cannot be a rainbow solution and such solutions will be called \emph{degenerate}, i.e.~the only solutions that will be considered are \emph{non-degenerate}.  
    
    The \emph{rainbow number} of $[n]$ for equation $eq$, denoted $\aS([n],eq)$, is the smallest $r$ such that every exact $r$-coloring of $[n]$ is guaranteed to have a rainbow solution to $eq$.  A coloring is \emph{extremal}, with respect to $eq$ and $[n]$, if it uses $\aS([n],eq) - 1$ colors and avoids rainbow solutions.  As a technical note, if $eq$ has $m$ variables and $m > n$ or there are no solution sets to $eq$ in $[n]$, define $\aS([n],eq) = n+1$.

    Some more notation and a convention about the way $[n]$ is colored are now introduced.  Let $c$ be an exact $r$-coloring on $[n]$.  Define $\mathcal{C}_i = \{ a\in [n]\, | \, c(a) = i\}$ and define $s_i \in \mathcal{C}_i$ such that $s_i$ is the smallest element of $\mathcal{C}_i$ for each color $i$.  Note that for any exact $r$-coloring $c$, it is always possible to have $s_i < s_j$ for $i < j$.  If that is not the case, say $s_i > s_j$ and $i<j$, an isomorphic coloring can be created by swapping the color of any number with color $i$ to have color $j$ and vice versa.
    
    Observation \ref{obs1} follows directly from the definition of rainbow numbers and describes the strategy used throughout the paper to find $\aS([n],eq)$.
    
    \begin{obs}\label{obs1}
       {\rm  If $c$ is an exact $(r-1)$-coloring of $[n]$ that avoids rainbow solutions for $eq$, then $r \le \aS([n],eq)$.  If every exact $r$-coloring of $[n]$ guarantees a rainbow solution to $eq$, then $\aS([n],eq) \le r$.}
    \end{obs}

    \begin{ex}{\rm
       Consider the equation $eq: x_1 + x_2 = x_3$ over $[5]$.  Define $c:[5] \to [3]$ by $c(1) = 1$, $c(2) = 2$, $c(3) =1$, $c(4) = 3$ and $c(5) = 1$.  Since this is an exact $3$-coloring that avoids rainbow solutions to $eq$, Observation \ref{obs1} gives $4 \le \aS([5],eq)$.  On the other hand, any exact $4$-coloring of $[5]$ must either use $4$ distinct colors in $\{1,2,3,4\}$ or $\{2,3,4,5\}$.  In the former case $\{1,2,3\}$ is a rainbow solution, and in the latter case $\{2,3,5\}$ is a rainbow solution.  Thus $\aS([5],eq) \le 4$.}

\end{ex} 
    
\section{The rainbow number of $[n]$ for the equation $x_1+x_2=x_3$}\label{sec:2}

This section focuses on the rainbow number for the equation $x_1+x_2 = x_3$ so for the rest of the section it is assumed that $eq$ is $x_1+x_2 = x_3$.  Also, when $\{a,b,c\}$ is presented as a solution to $eq$ it will be recognized that $x_1=a$, $x_2 = b$ and $x_3 = c$.  In order to discuss the results in this section it is easier if an exact $r$-coloring uses the color set $\{0,1,\dots,r-1\}$.  Lemma \ref{lower1} starts with a specific coloring that avoids rainbow solutions.

    \begin{lem}\label{lower1} 
        For $n \ge 3$, $\left\lfloor\log_2(n) + 2\right\rfloor \le \aS([n],eq)$.
        
        \bpf
        The number of trailing zeros that an integer has is the number of zeros, starting from the right, before there is a nonzero digit.  Define an exact $(\left\lfloor \log_2(n) + 1\right\rfloor)$-coloring of $[n]$ as
        $$
        c(x) = \text{the number of trailing zeros in the binary representation of $x$}
        $$
        (see Example \ref{ex:trailing}). Let $a,b\in [n]$. If $a$ and $b$ are odd, then $c(a) = c(b) = 0$ since their binary representation ends with a $1$.  This means $\{a,b, a+b\}$ is not a rainbow solution. If $a$ is odd and $b$ is even, then $a+b$ is odd so $c(a) = c(a+b) = 0$, thus $\{a,b, a+b\}$ is not a rainbow solution. Finally, consider the case where $a$ and $b$ are both even. If $a$ and $b$ have the same number of trailing zeros, then $c(a) = c(b)$ meaning $\{a,b, a+b\}$ is not a rainbow solution. If $c(a) \ne c(b)$ assume, without loss of generality, that $c(a) < c(b)$.  Since the number of trailing zeros of $b$ exceeds the number of trailing zeros of $a$, it follows, via binary arithmetic, that $b+a$ has the same number of trailing zeros as $a$.  Thus, $c(a+b) = c(a)$ and $\{a,b, a+b\}$ is not a rainbow solution. Therefore, no rainbow solutions exist with this coloring, and $\left\lfloor\log_2(n) + 2\right\rfloor \le \aS([n],eq)$.\epf
    \end{lem}
    
    \begin{ex}\label{ex:trailing}

    {\rm A table describing the coloring from Lemma \ref{lower1}.}
    $$
    \begin{array}{|c|c|c|}
    \hline
        x & x \text{ in binary} & c(x)\\
    \hline
    1 & 1 & 0\\
    \hline
    2 & 10 & 1\\
    \hline
    3 & 11 & 0\\
    \hline
    4 & 100 & 2\\
    \hline
    5 & 101 & 0\\
    \hline
    6 & 110 & 1\\
    \hline
    7 & 111 & 0 \\
    \hline
    8 & 1000 & 3 \\
    \hline
    
    \end{array}$$

\end{ex}
    

Lemma \ref{upperhelp1} indicates that when attempting to color $[n]$ and avoid rainbow solutions, there are restrictions on how quickly new colors can be added.

    \begin{lem}\label{upperhelp1}
        Let $c$ be an exact $r$-coloring of $[n]$, with color set $\{0,1,\dots, r-1\}$, that avoids rainbow solutions, then
        \begin{enumerate}
            \item if $s_i = \ell$, then $2\ell \le s_{i+1}$ for $0 \le i \le r-2$,
            \item $2^i \le s_i$ for $0 \le i \le r-1$.
        \end{enumerate}
        
        \bpf
        For the first claim, if $\ell = 1$, then $i = 0$, so the smallest $m$ for which $s_2 = m$ is $m = 2$. If $\ell \ge 2$, let $1 \le a < \ell$ with $a+\ell \le n$. Then $c(a) \ne c(\ell)$, so $c(a + \ell)\in \{c(a), c(\ell)\}$. Thus, $c(a + \ell) \ne i + 1$, hence $2\ell \le s_{i+1}$.
        
        Proving the second claim proceeds by induction on $k$.  The base case is easily observed, in particular, $s_0 = 1 = 2^0$. For the induction hypothesis, assume $2^k \le s_k$ for $0 \le k \le r-2$.  Applying part 1 to our induction hypothesis yields $2(2^k) = 2^{k+1} \le s_{k+1}$ which completes the proof.\epf
    \end{lem}
    

Theorem \ref{result1} follows almost directly from Lemmas \ref{lower1} and \ref{upperhelp1}.

    \begin{thm}\label{result1}
        For $n \ge 3$, $\aS([n],eq) = \left\lfloor\log_2(n) + 2 \right\rfloor$.
        
        \bpf
        By Lemma \ref{lower1} $\left\lfloor\log_2(n) + 2\right\rfloor \le \aS([n],3)$.  Let $r = \left\lfloor\log_2(n) + 2\right\rfloor$ and $c$ be an exact $r$-coloring, with color set $\{0,1,\dots,r-1\}$, of $[n]$ that avoids rainbow solutions. Then, by Lemma \ref{upperhelp1}, $2^{r-1} \le s_{r-1}$, so
        $$
        2^{r-1} = 2^{\left\lfloor\log_2(n) + 1\right\rfloor} \le s_{r-1}.$$

        
        Thus, $n < s_{r-1}$.  So if color $r-1$ appears in a coloring of $[n]$ there must be a rainbow solution, a contradiction that such a coloring exists.  Therefore, $\aS([n],eq) = \left\lfloor\log_2(n) + 2 \right\rfloor$.\epf
    \end{thm}
    

\section{The rainbow number of $[n]$ for the equation $x_1 + x_2 + x_3 = x_4$}\label{sec:3}



This section defines $eq: x_1 + x_2 + x_3 = x_4$ (except in Theorem \ref{lem:genlower}) and an exact $r$-coloring will use color set $[r]$.  Lemma \ref{lem:mins1} is used to define a coloring on $[n]$ that avoids rainbow solutions which gives a lower bound on $\aS([n],eq)$. It is interesting to note that the equations in Section \ref{sec:2} and \ref{sec:3} are similar but one result is logarithmic and the other is linear.

    \begin{lem}\label{lem:mins1}


    For $\ell \ge 2$, let \{$b_1,b_2 \ldots, b_{\ell}\} \subseteq [n]$, with $b_i < b_{i+1}$ for $1 \le i \le \ell -1$, such that $\sum\limits_{i=1}^{\ell} b_i \geq n$.  The smallest $b_{\ell}$ that satisfies these conditions gives a corresponding maximum $b_1$ value of $$b_1 = \left\lceil\frac{2n-\ell(\ell-1)}{2\ell}\right\rceil.$$
\end{lem}
\begin{proof}
    Notice that by first minimizing $b_{\ell}$ and then maximizing $b_1$ gives $b_{i+1} = b_{i}+1$, for $1 \le i \le \ell - 1$, and $b_{\ell} = b_1+\ell-1$.
     Observe,

    
        $$\sum\limits_{i=1}^{\ell} b_i \geq n. $$
        
            Well-known facts about triangular numbers gives the equivalent inequality
            
            $$\frac{b_{\ell}(b_{\ell}+1)}{2}-\frac{(b_1-1)b_1}{2} \geq n,$$
            
        
        and solving for $b_1$ yields
        
        

        $$b_1 \geq \frac{2n - \ell(\ell-1)}{2\ell}.$$
     Since $b_1$ is an integer, and the $b_i$'s are consecutive, it can be concluded that

       $$ b_1 = \left\lceil \frac{2n - \ell(\ell-1)}{2\ell} \right\rceil. $$ \end{proof}
  
   
    \begin{thm}\label{lem:genlower}

Let $eq: \sum_{i=1}^{k-1} x_i = x_k$, $k \ge 4$ and define $L = \left\lceil\frac{2n-\ell(\ell-1)}{2\ell}\right\rceil$ for $\ell = k-2$.  Then

$$\aS([n],eq) \ge  
        \left\{\begin{array}{ll}
            n+1 & \text{if } n<\frac{(k-1)k}{2}, \\
            n-L+3  & \text{otherwise.}
      
\end{array}
\right.$$

\end{thm}

\begin{proof} The lower bound of $n+1$ (which is actually an equality) occurs when $[n]$ does not have any solutions to $eq$ and is based on $\sum_{i=1}^{k-1} i = \dfrac{(k-1)k}{2}$. Hence every integer in $[n]$ can be colored distinctly and $\aS([n],eq) = n+1$ by definition.  If $[n]$ does have solutions to $eq$, then consider the coloring $c$ on $[n]$ where

$$c(x) = \left\{\begin{array}{cc}

    1 & \text{ if $1 \le x \le L-1$,}\\
    
    i+2 & \text{ if $x = L+i$.}\\

\end{array} \right. $$

This coloring avoids rainbow solutions since any rainbow solution must have at least $k-1$ elements that are bigger than $L$ and Lemma \ref{lem:mins1} indicates that the sum of these elements is greater than $n$.  Thus, no rainbow solutions exist.

It remains to count how many colors are used.  The number of elements that are not colored $1$ is $n-L+1$, thus $c$ uses $n-L+2$ colors.  Since $c$ avoids rainbow solutions, this gives $n-L+3 \le \aS([n],eq)$.  \end{proof}

     

    \begin{cor}\label{lower2}
        For $n \ge 5$, $\left\lfloor\dfrac{1}{2}(n+7)\right\rfloor \le \aS([n],eq)$.
    \end{cor}
    
    \bpf As noted in the introduction to Section \ref{sec:3}, $eq$ in this problem has $k=4$ and $\ell = 2$, thus so $L = \left\lceil \frac{n -1}{2} \right\rceil$.  Applying Theorem \ref{lem:genlower} gives the result. \epf

    \begin{prop}\label{uniquelower1}
        If $n$ is odd and $n\ge5$, then $\aS([n],eq) = \dfrac{1}{2}(n+7)$ and there is a unique extremal coloring of $[n]$.
    \end{prop}
    
    \begin{proof} When $n=5$ the result is immediate.  When $n=7$ the lower bound is a result of Corollary \ref{lower2}.  The upper bound is trivial since the only coloring of $[7]$ with seven colors is to color each number distinctly.  If $c'$ is an extremal coloring of $[7]$, then $6$ colors are used. Since $\{1,2,3,6\}$ and $\{1,2,4,7\}$ are both solutions, $c'(1) = c'(2)$ and the coloring is unique.
    
    For the induction hypothesis, assume that for all odd $k$ with $7\le k \le n$, $n$ odd, that $\aS([k],eq) = \dfrac{1}{2}(k+7)$ and there is a unique extremal coloring of $[k]$, namely the coloring provided in Theorem \ref{lem:genlower}.  Define $r = \dfrac{1}{2}((n+2)+7)$, $L = \dfrac{n-1}{2}$ and let $c$ be an exact $r$-coloring of $[n+2]$ with color set $\{1,\dots,r\}$.  The induction hypothesis implies that if $r$ or $r-1$ colors appear in $[n]$, then there is a rainbow solution, so $s_{r} = n+2$ and $s_{r-1} = n+1$.  The induction hypothesis also implies that if $r-2$ colors appear in $[n]$, there is a unique coloring of $[n]$ with $r-2$ colors that avoids rainbow solutions.  However, applying the unique coloring yields $c(1)= 1$, $c(L) = 2$, $c(L+1) = 3$ and $c(n+1) = r-1 > 3$, thus $\{1,L,L+1 ,n+1\}$ is a rainbow solution.   Therefore, $\aS([n+2],eq) \le \dfrac{1}{2}(n+9)$.  Equality comes from the lower bound in Corollary \ref{lower2}.
    
    Now let $c$ be an exact $(r-1)$-coloring of $[n+2]$, with color set $\{1,2,\dots,r-1\}$, that avoids rainbow solutions.  If $r-1$ colors appear in $[n]$ the inductive hypothesis implies there is a rainbow solution, thus $n+1 \le s_{r-1} \le n+2$.
    
    If $r-2$ colors appear in $[n]$, then $[n]$ has the coloring from Theorem \ref{lem:genlower}.  This means $s_2 = L$ and $s_3 = L+1$.  Hence, either $\{1,L,L+1,n+1\}$ or $\{2,L,L+1,n+2\}$ is a rainbow solution since $n+1 \le s_{r-1} \le n+2$.  Therefore, $n < s_{r-2}$ which implies $s_{r-2} = n+1$ and $s_{r-1} = n+2$.
    
    Now consider the case where $r-3$ colors appear in $[n]$.  
    If $r-j-1$ colors appear in $[n-2j+2]$, for $2\le j \le \dfrac{n+3}{4}$, then the coloring from Theorem \ref{lem:genlower} gives $L_j = \dfrac{n-2j+2-1}{2} = L-j+1 = s_2$, $L-j+2 = s_3$, $\dots$ ,$n-2j +2 = s_{r-j-1}$.  However, notice that $$L-j+1 \le L < L+1 \le n-2j+2 < n+1,$$ thus $\{1,L,L+1,n+1\}$ is a rainbow solution which implies $n-2j+3 \le s_{r-j-1}$.
    
    Combining this with $s_i < s_{i+1}$ gives
    
    \begin{equation}\tag{*} n-2j+3 \le s_{r-j-1} \le n-j+2.\label{ineq:*}\end{equation}
    
    Note that Inequality (\ref{ineq:*}) gives $s_2 =s_{r-(r-3)- 1} \le n-(r-3)+2 = L+1$.  The remainder of the proof will show that if $s_2 \neq L+1$, then there is a rainbow solution.  This then implies that $s_2 = L+1$ and $c$ must be the coloring from Theorem \ref{lem:genlower}.
    
    Set $[n+2] = \{1\} \cup A \cup B \cup \{n-1,n,n+1,n+2\}$ with $A = \{2,3,\dots,L\}$ and $B = \{L+1,L+2,\dots,n-2\}$.  Note that $|A| = |B| = L-1 = \dfrac{n-3}{2}$ and that $s_2\in A$.  Also, for any $x\in A$, the set $\{1,x,n-x,n+1\}$ is a solution to $eq$ and $c(1) = 1\neq r-2 = c(n+1)$.  This implies that each element of the pair $\{x,n-x\}$ is either the same color or at least one of them is color $1$.  Further, since $x\in A$ and $n-x\in B$, $|c(A\cup B)|$ uses at most $L-1$ of the remaining $|\{2,3,\dots,r-3\}| = r-4 = L+1$ colors.  This forces $c(n-2) = r-4$, $c(n-1) = r-3$ and each pair $\{x,n-x\}$ contributing one color from $\{2,3,\dots,r-5\}$ that is not accounted for among any other $\{y,n-y\}$ pair for $y\in A$.  However, now $\{1,s_2,n-s_2+1,n+2\}$ is a rainbow solution.  To see this, observe that $c(s_2-1) = 1$ and if $c(n-s_2+1) = 2$, then the $\{s_2-1,n-s_2+1\}$ pair contributes the color $2$ so the $\{s_2,n-s_2\}$ pair must contribute a color that is not $2$. This means $c(n-s_2) \neq 2$ which makes $\{1,s_2,n-s_2,n+1\}$ a rainbow solution. \end{proof}

    \begin{lem}\label{upper1}
        If $n\ge 5$ is odd, then $\aS([n+1],eq) \le \dfrac{1}{2}(n+7)$.
        
        \bpf
       Let $r = \dfrac{1}{2}(n+7)$ and $c$ be an exact $r$-coloring of $[n+1]$. If $r$ colors appear in $[n]$, then, by Proposition \ref{result3}, there exists a rainbow solution. Thus $c(n+1) = r$ and $r-1$ colors appear in $[n]$.  Note that there is unique extremal coloring of $[n]$ with $r-1$ colors, thus $\left\{1,\dfrac{n-1}{2},\dfrac{n+1}{2}, n+1\right\}$ is a rainbow solution since $c(1) = 1$, $c((n-1)/2) =2$, $c((n+1)/2) = 3$ and $c(n+1) = r$. Therefore, $\aS([n+1],eq) \le \dfrac{1}{2}(n+7)$.\epf
    \end{lem}
    

Theorem \ref{result3} follows Proposition \ref{uniquelower1} and Lemma \ref{upper1}.
    \begin{thm}\label{result3}
        For $n \ge 5$, $\aS([n],eq) = \left\lfloor\dfrac{1}{2}(n+7)\right\rfloor$.
    \end{thm}
    


\section{Conclusion}

It is believed that future work on equations of the form $\sum_{i=1}^{k-1} x_i = x_k$, $k \ge 5$ should focus on the coloring exhibited in Theorem \ref{lem:genlower} and arguments should utilize $s_i$.  Data indicates that there are unique colorings, when $k=5$, for $n=9,12,15,\dots$  and colorings for other values of $n$ are closely related to the unique colorings.  Other questions to consider are more general equations like $\dsp\sum_{i=1}^k a_ix_k = b$.  It is not clear which, if any, techniques from this paper will be applicable in the more general situation.

\section*{Acknowledgements} 
Thank you to the referees for the helpful comments that have improved the paper, in particular for the simplification of the proof of Proposition \ref{uniquelower1}.  Thanks also to the University of Wisconsin-La Crosses's Undergraduate Research and Creativity grants which helped support this research.

\end{document}